\documentclass{scrartcl}

\pdfoutput=1 % per arXiv

\usepackage{iftex}

\ifPDFTeX 
  \usepackage[utf8]{inputenc}
  \usepackage[T1]{fontenc}
  \usepackage{lmodern}
\else
  \usepackage[no-math]{fontspec}
\fi

\usepackage[mathjax,latexmk]{lwarp}
\HTMLLanguage{en-GB}
\CSSFilename{custom.css}

\warpHTMLonly{\usepackage{bootstrap}}

\usepackage[british]{babel}
\usepackage{etoolbox,xifthen,xparse}

\usepackage[dvipsnames]{xcolor}
\usepackage{aliascnt,enumitem}

\usepackage{amsmath}
\usepackage{amsfonts,amssymb,amsthm,mathtools,booktabs}
\ifPDFTeX
  \usepackage{bbm}
  \usepackage{mathrsfs}
\else
  \usepackage{unicode-math}
  
\fi
\usepackage[retainorgcmds]{IEEEtrantools}

\usepackage{url,doi}
\usepackage[numbers]{natbib}

\usepackage{hyperref,bookmark}

\hypersetup{hidelinks}

\newcommand{\newsstheorem}[2]{
  \newaliascnt{#1}{dummy}
  \newtheorem{#1}[#1]{#2}
  \aliascntresetthe{#1}
  \expandafter\def\csname #1autorefname\endcsname{#2}
}
\theoremstyle{plain}
\newsstheorem{theorem}{Theorem}
\newsstheorem{proposition}{Proposition}
\newsstheorem{corollary}{Corollary}
\newsstheorem{lemma}{Lemma}
\theoremstyle{definition}
\newsstheorem{definition}{Definition}
\newsstheorem{example}{Example}
\newsstheorem{assumption}{Assumption}
\theoremstyle{remark}
\newsstheorem{remark}{Remark}

\setlist[enumerate,1]{label={(\roman*)}}
\setlist[enumerate,2]{label={(\alph*)}}
\setlist[enumerate,3]{label={(\Roman*)}}

\newcommand*\printandhtml[1]{%
  #1%
  \CustomizeMathJax{#1}%
}

\printandhtml{
  \newcommand\dd{\mathrm{d}}
  \newcommand\NN{\mathbb{N}}
  \newcommand\RR{\mathbb{R}}
  \newcommand\CC{\mathbb{C}}
  \newcommand\PP{\mathbb{P}}
  \newcommand\EE{\mathbb{E}}
  \newcommand\Xx{\mathcal{X}}
  \newcommand\Mm{\mathcal{M}}
  \newcommand\GG{\mathscr{G}}
  \newcommand\Zb{\mathbf{Z}}
  \newcommand\Xb{\mathbf{X}}
  \newcommand\pf{\mathfrak{p}}
  \newcommand\Indic[1]{\Ind_{\{#1\}}}
  \newcommand\from{\colon}
  \newcommand\iu{\mathrm{i}}
  \newcommand\crc{^\ast}
  \newcommand\tree{\mathbb{U}}
  \newcommand\eve{\varnothing}
  \DeclareMathOperator{\dom}{dom}
  \DeclareMathOperator{\sgn}{sgn}
}
\newcommand\Ind{\mathbbm{1}}
\CustomizeMathJax{\newcommand\Ind{\mathbb{1}}}

\providecommand{\email}[1]{\href{mailto:#1}{\nolinkurl{#1}}}
\providecommand{\arxivref}[2]{\href{http://arxiv.org/abs/#1}{arXiv:#1}%
  \ifblank{#2}{}{ [#2]}%
}

\title{A growth-fragmentation model connected to the ricocheted stable process}
\author{A.\ R.\ Watson\footnote{University College London, UK. \email{alexander.watson@ucl.ac.uk}}}
\date{\today}

\makeatletter
\hypersetup{
  pdftitle={\@title},
  pdfauthor={A. R. Watson},
  pdflang={en-GB}
}
\makeatother

\begin{document}
\maketitle

\begin{abstract}
  Growth-fragmentation processes describe the evolution
  of systems in which cells grow slowly and fragment suddenly.
  Despite originating as a way to describe biological phenomena,
  they have recently been found to describe
  the lengths of certain curves in
  statistical physics models.
  In this note, we describe a new growth-fragmentation process
  connected to random planar maps with faces of large degree, having
  as a key ingredient the ricocheted stable process recently
  discovered by Budd.
  The process has applications to the excursions of
  planar Brownian motion and Liouville quantum gravity.
  
\end{abstract}

\section{Introduction}

A random planar map is a random graph embedded in the plane. One approach
to understanding features of random planar maps is to explore
them, starting from a distinguished `root face' and sequentially discovering
adjoining faces at random.
Part-way through such an exploration, the boundary of the explored
region forms a set of cycles of vertices, and subsequent
exploration may either enlarge certain cycles
by adding newly discovered vertices, or cause them to
split apart as a face is discovered which cuts across one.
\citet{BBCK-maps} considered exploring a Boltzmann 
random planar map
with faces of large degree and, in a scaling limit
as the degree of the root face grows, they found and characterised
a growth-fragmentation process which describes the growth
and splitting of perimeters of these cycles.

Denote Bertoin et al.'s growth-fragmentation process by $\Zb\crc = (\Zb\crc(t): t\ge 0)$,
where
$\Zb\crc(t) = (Z_1\crc(t),Z_2\crc(t),\dotsc)$ 
describes the collection of cell `masses'
(rescaled cycle perimeters)
at 
time $t$, listed in decreasing order.
The cell masses $\Zb\crc$ have both upwards and downwards jumps.
Whenever a cell of $\Zb\crc$ jumps
down in mass from $u$ to $ux$ (with $0<x<1$), an additional
cell of mass $u(1-x)$ is added to the system (thus conserving mass at downward jumps,
i.e. fragmentation events). However, when a cell mass jumps up, no additional cells are added.

In this work, we consider the effect of additionally introducing
cells
at the \emph{upward} jump times of
cells in $\Zb\crc$.
If a cell jumps in mass from $u$ to $ux$, where $x>1$,
then with probability $r$,
we introduce a new cell of mass $u(x-1)$.
This new growth-fragmentation will be denoted $\Zb$.

In terms of the exploration of a Boltzmann map,
the downward jumps of cell masses in $\Zb\crc$
come from splitting of cycles, whereas
the upward jumps represent the discovery of faces of macroscopically large degree.
Heuristically, we may think of each new cell arising
from the upward jumps in $\Zb$ as corresponding to
the discovery,
with probability $r$, of a further conditionally
independent Boltzmann map
embedded within a face of the original map.

In \cite{Budd-ricochet}, Budd studied the metric properties of random planar
maps coupled with an $O(n)$ loop model.
Such maps can be decomposed along the loops to form a \emph{gasket},
which is another random planar map
of the type studied by \cite{BBCK-maps},
together with conditionally independent submaps
to be inserted into certain faces of the gasket.
It was
found that the scaling limit of perimeters of certain distinguished cycles
in such maps is related to the so-called \emph{ricocheted stable process}.
This is a perturbation of the stable process in which, at each attempted jump
below zero, the process is `ricocheted' with some probability
and restarts from a positive value.

\medskip\noindent
The main result of this note,
\autoref{t:gf-rico} in \autoref{s:results},
is that the ricocheted stable process and
its conditionings are intimately connected to the process
$\Zb$,
and that the ricochet probability is proportional to the probability $r$ 
of discovering a new cell at each upward jump.
This is natural, given the similarity in the descriptions
of the growth-fragmentation and random planar map model
above, but the connection does not appear to be known.
A rigorous description of the appearance
of this process in the scaling limit of loop-decorated random maps may
be beyond reach at present, so we confine ourselves here
to the discussion of the growth-fragmentation processes
themselves.

As it turns out, the same process $\Zb$ can be used to describe
the excursions of a planar Brownian motion, as explored
by \cite{AdS-exc}.
We discuss this in more detail
in \autoref{s:results}, and give an application to functionals
of small excursions and, in the general setting, large cells.

Similarly, the process $\Zb\crc$ arises in the context of
Liouville quantum gravity \cite{MSW-scle}, 
where the idea of introducing new cells at the upward jumps
is also discussed,
giving rise to the tree structure
denoted there by $\tilde{\mathcal{T}}$. This tree
can be viewed as representing
the genealogy of the growth-fragmentation $\Zb$.

\section{Ricocheted stable processes}
\label{s:ricochet}

We begin with a precise description of the ricocheted stable process,
which is a self-similar Markov process derived from a Lévy process.
Lévy processes are real-valued
stochastic processes with stationary, independent
increments and initial value $0$.
A Lévy process $X$ can be characterised by its Laplace exponent
$\psi$,
which is defined by the relation $\EE[ e^{qX_t}] = e^{t\psi(q)}$
and satisfies the \emph{Lévy--Khintchine formula}
\[
  \psi(q) 
  = \frac{1}{2}\sigma^2 q^2
  + aq 
  + \int_{\RR} (e^{qy} - 1 - qy\Indic{\lvert y\rvert \le 1}) \, \Pi(\dd y),
\]
where $\sigma\ge 0$, $a\in \RR$, and $\Pi$ is a measure (the Lévy measure)
on $\RR\setminus\{0\}$
satisfying $\int_{\RR} \min\{y^2,1\}\, \Pi(\dd y) < \infty$.
In general, $\psi$ is finite at least for $q \in \iu\RR$, but in many
cases it is finite even for some neighbourhood of $0$ in $\CC$.
A Lévy process with initial value $0$ can be started from
any $x\in\RR$ by considering $t \mapsto X_t + x$.

Consider a stochastic process $X$ with associated probability measures
$(\PP_x)_{x\in E}$, for some state space $E$ which is closed under
multiplication by positive scalars,
where $\PP_x(X_0=x) = 1$.
We say that $X$ satisfies the \emph{scaling property}
with index $\theta$ if, for any $x \in E$ and any $c>0$,
\begin{equation}
  \label{e:scaling}
  (cX_{tc^{-\theta}})_{t\ge 0}
  \text{ under } \PP_x \text{ has the same distribution as }
  X \text{ under } \PP_{cx}.
\end{equation}

Let parameters $\theta,\rho$ be chosen from the set
\begin{multline*}
  \mathcal{A}_{\theta,\rho}
  =
  \{ (\theta,\rho) : \theta \in (0,1), \rho \in (0,1)\}
  \cup \{ (\theta,\rho): \theta \in (1,2);\, \theta\rho,\theta(1-\rho) \in (0,1) \} \\
  {} \cup \{ (\theta,\rho) = (1,1/2) \}.
\end{multline*}
Then, the function
\begin{equation}
  \label{e:stable}
  \psi_{\theta,\rho}(q)
  = -c\lvert q\rvert ^\theta \bigl( 1-\iu\beta \tan(\pi\theta/2) \sgn(q/\iu)\bigr),
  \qquad
  q \in \iu\RR,
\end{equation}
where $\beta = \frac{\tan(\pi \theta (\rho-1/2))}{\tan(\pi\theta/2)}$,
is the Laplace exponent of a Lévy process, known as a \emph{stable process}
with index $\theta$ (or a \emph{$\theta$-stable process})
and positivity parameter $\rho$.

Stable processes fulfil the scaling property \eqref{e:scaling}
with $E=\RR$.
Indeed, every Lévy process satisfying the scaling property
has Laplace exponent of form \eqref{e:stable}, up to a multiplicative factor
and with a few exceptions:
our parameter restrictions mean that 
in this work we neglect symmetric Cauchy processes with drift ($\theta=1,\, \rho\ne 1/2$),
Brownian motion ($\theta=2$),
processes with monotone paths ($\theta\in (0,1)$, $\rho =0$ or $\rho=1$)
and other processes which jump only in one direction
($\theta\in (1,2)$, $\theta\rho = 1$ or
$\theta(1-\rho)=1$.)
%% Note to self.
% We exclude one-sided jumps because the same is done in KPV. However,
% unlike in the context of the path-censored processes, I think there
% is a case to be made for considering the cases of downwards-only
% jumps. There, the ricocheted process is not the same as the killed
% process. I think the Laplace exponent should be the same. Possibly
% the WHF differs? However, in this case we don't see it arising
% from our growth-fragmentation.

We are now in a position to define the ricocheted stable process,
as detailed by \citet{Budd-ricochet} and \citet{KPV-dblhg}.
Let $x>0$ and take a stable process $X$ under $\PP_x$ with index $\theta$
and positivity parameter $\rho$. Take an additional parameter
$\pf \in [0,1]$, the \emph{ricochet probability}.
Define the first passage time below zero:
\[ \tau = \inf\{t\ge 0: X_t < 0 \}. \]
We define a new process $Y$, which follows $X$ up until time $\tau$.
At time $\tau$, with probability $1-\pf$, $Y$ is killed
(i.e., sent to the absorbing state $0$)
and with probability $\pf$,
it moves not to $X_\tau$, but to $Y_\tau = -X_\tau$; the
jump of the stable process below zero is `ricocheted' back above the
origin. Subsequently, $Y$ follows a stable process under
$\PP_{-X_{\tau}}$, until the first passage time of this process below zero,
at which point the probabilistic ricochet occurs again. This iteratively
constructs an entire path of the process $Y$, up until possibly reaching 
the point $0$, at which point $Y$ is killed.

The process $Y$ is called the \emph{ricocheted stable process}.
It is an example of a 
\emph{positive, self-similar Markov process (pssMp)}.
In general, a process is called a pssMp with index $\theta$
if it is a standard Markov
process with state space $[0,\infty)$, has $0$ as an absorbing state,
and satisfies the scaling property \eqref{e:scaling}.

An important property of pssMps, which explains their usefulness and
will allow us to explore the ricocheted stable process in more detail,
is that they are in bijection with Lévy processes, via the
\emph{Lamperti transformation} \cite[§13]{Kyp2}. 
For a pssMp $Y$ with index $\theta$,
let $S(t) = \int_0^t (Y_u)^{-\theta}\, \dd u$, and define
$T$ as the inverse of $S$. Then, the process
\[ \xi_s = \log Y_{T(s)}, \quad s \ge 0, \]
is a Lévy process (possibly killed and sent to $-\infty$),
the Lamperti transformation of $Y$.

The Lamperti transformations of the pssMp given by the
stable process killed at first passage below zero, and conditionings
thereof,
are well-known \cite{KP-hg}, and belong
to the class of `hypergeometric' Lévy processes.
This identifiability extends to ricocheted stable processes as well
\cite{Budd-ricochet,KPV-dblhg}: the Lamperti transformation
of the ricocheted stable process $Y$ is the Lévy
process $\xi$ with Laplace exponent
\[
  \psi(q)
  =
  - 2^{\theta}
  \frac{\Gamma(\frac{\theta-q}{2})\Gamma(\frac{1+\theta-q}{2})}
  {\Gamma(\frac{b-\sigma-q}{2})\Gamma(\frac{2-\sigma-b-q}{2})}
  \frac{\Gamma(\frac{1+q}{2})\Gamma(\frac{2+q}{2})}
  {\Gamma(\frac{\sigma+b+q}{2})\Gamma(\frac{2+\sigma-b+q}{2})},
  \quad
  q \in (-1,\theta),
\]
where $\sigma = \frac{1}{2} - \theta(1-\rho)$
and $b = \frac{1}{\pi}\arccos(\pf\cos(\pi\sigma))$.

We will say that $Y$ is a \emph{ricocheted stable process
with parameters} $(\theta,\sigma,b)$. If we specify the
set of admissible parameters
\begin{multline*}
  \mathcal{A}_{\theta,\sigma,b} =
  \{ (\theta,\sigma,b) : \theta \in (0,1), \,
  \sigma \in (1/2-\theta,1/2), \, b \in [\lvert \sigma \rvert, 1/2] \} \\
  {} \cup
  \{ (\theta,\sigma,b) : \theta \in (1,2),\,
    \sigma \in (-1/2,3/2-\theta),\,
  b \in [\lvert\sigma\rvert,1/2] \} \\
  {} \cup \{ (\theta,\sigma,b) : \theta=1, \sigma=0, b\in [0,1/2]\},
\end{multline*}
then these are in one-to-one correspondence with the admissible
parameters $\theta$, $\rho$ and $\pf$.

The Lamperti transform of $Y$
has other nice properties. Its Wiener-Hopf factorisation is known,
and it is an example of a \emph{double hypergeometric Lévy process}
\cite{KPV-dblhg}.

Note that when $\rho = 1/2$, meaning that $X$ is symmetric, the
ricocheted stable process is equal to $\lvert X\rvert$, the absolute value
of $X$, with additional probability of
killing at the times when $X$
changes sign.
When $\pf = 0$, $Y$ is the stable process killed upon first passage below zero.

\section{A novel growth-fragmentation}
%\medskip\noindent
$\Zb\crc$ is an example of a \emph{Markovian
growth-fragmentation}, in the sense of \cite{BeMGF}.
These processes can be described more formally as
follows.
For each $x>0$, we construct a process, under the probability
measure $\PP_x$, starting from a single cell of mass $x$.
For $u \in \tree = \cup_{n\ge 0}\NN^n$, the
set of Ulam-Harris labels, we define processes
$\Xx_u = (\Xx_u(t))_{t\ge 0}$ as follows.
The process $\Xx_\eve$ is a Markov process 
with state space $[0,\infty)$ and with prescribed distribution,
known as the \emph{cell process},
which we interpret as the mass of the initial cell,
such that $\Xx_\eve(0) = x$.
We assign $b_\eve = 0$, its birth time.
If we have constructed $\Xx_u$, then we list the jump times and 
jump ratios
$\{ (t_i,x_i) : i \ge 1 \}$ in some appropriate order, such as
in decreasing order of jump ratio; recall
here that, if $\Xx_u$ jumps at time $t_i$, the
jump ratio is $x_i = \frac{\Xx_u(t_i)}{\Xx_u(t_i-)}$.
% relative
% jump size is $\frac{\Xx_u(t)-\Xx_u(t-)}{\Xx_u(t-)}$.
A \emph{selection probability} $c\from [0,\infty) \to [0,1]$
determines which jumps lead to additional cells.
For each $i\ge 1$, with probability $c(x_i)$, we define the process
$\Xx_{ui}$ to be equal in distribution to the cell process, started
from $\Xx_{ui}(0) = \lvert\Xx_u(t_i)-\Xx_u(t_i-)\rvert = \lvert\Xx_u(t_i-)(x_i-1)\rvert$,
and let $b_{ui} = b_u + t_i$.
(On the complementary event with probability $1-c(x_i)$,
no new process is defined.)
The
Markovian growth-fragmentation process at time $t\ge 0$,
$\mathbf{X}(t)$, is then given by
the decreasing rearrangement $(X_1(t),X_2(t),\dotsc)$
of $(\Xx_u(t - b_u) : u \in \tree)$.
For clarity, we will call $\mathcal{X} = (\mathcal{X}_u : u \in\tree)$
the generational
growth-fragmentation associated with $\mathbf{X}$.

When the cell process in a Markovian growth-fragmentation is 
a pssMp
of index $\theta$, we say that the growth-fragmentation
itself is \emph{self-similar}, with \emph{index} $\alpha=-\theta$.
This is the case for
the process $\Zb\crc$ found by \cite{BBCK-maps},
parametrised by $\theta \in (1/2,3/2]$,
in which the cell process
is a pssMp of index $\theta$
and can be characterised
by its Lamperti transformation,
which is a Lévy process whose Laplace exponent
is given \cite[p.~702]{BBCK-maps} by
\[
  \psi\crc(q)
  = aq + \int_{(-\log 2,\infty)} \bigl(e^{qy} - 1 +q(1-e^y)\bigr)\, \Lambda(\dd y),
\]
where
\[
  a
  = \frac{\Gamma(2-\theta)}{2\Gamma(2-2\theta)\sin(\pi\theta)}
  + \frac{\Gamma(\theta+1)B_{1/2}(-\theta,2-\theta)}{\pi},
\]
with $B$ the incomplete beta function,
and $\Lambda = \nu \circ \log^{-1}$, with $\nu$ defined by
\begin{multline*}
  \nu(\dd x)
  = \frac{\Gamma(\theta+1)}{\pi} \left(
    (x(1-x))^{-(\theta+1)} \Indic{1/2<x\le 1}
    \right. \\
    \left. {}
    + \sin(\pi(\theta-1/2)) (x(x-1))^{-(\theta+1)} \Indic{x>1}
  \right)\dd x,
\end{multline*}
known as the \emph{dislocation measure} of $\Zb\crc$.
In the case of $\Zb\crc$,
the selection probability is given by $c(x) = \Indic{x<1}$, meaning
that only negative jumps of the cell masses lead to
introduction of new cells.

%% Alternative description:
% Much like pssMps, a self-similar growth-fragmentation $\Xb$
% of index $\alpha$ can be transformed
% by a random time-change. For each cell $u \in \tree$, we define
% $S_u(t) = \int_0^t \Xb_u(v)^{\alpha} \, \dd v$, and $T_u$ as its inverse function.
% A new process, $\Xb\circ T$, can then be defined by
% \[
%   (\Xb\circ T)_u(s) = X_u(T_u(s)), \quad u \in \tree. 
% \]
% The effect of this time-change is to move the index of self-similarity
% from $\alpha$ to $0$. $\Xb\circ T$ is a Markovian growth-fragmentation
% whose cell process is given by taking the exponential of the Lamperti
% transform of $\Xb$'s cell process, and whose selection probability
% is the same as that of $\Xb$.
% 
% A self-similar growth-fragmentation of the type described above
% can be uniquely characterised by
% its index together with its \emph{cumulant}, defined by the formula
% \[ 
%   e^{s\kappa(q)} = \EE_1\left[\sum_{i\ge 1} (\Xb\circ T)_u(s)^q \right],
%   \quad s \ge 0,
% \]
% for $q \in \RR$ such that the right-hand side is finite.

A self-similar growth-fragmentation $\Xb$ can be 
characterised \cite[Theorem~1.2]{Shi-gf} by its index
$\alpha$ and its \emph{cumulant}, which is defined 
(see \cite[equation (10)]{BBCK-maps}) by the formula
\begin{equation}\label{e:kappa-gen}
  \EE_1\left[ \sum_{u\in\tree, \, \lvert u\rvert = 1} \Xx_u(0)^q \right]
  = 1 - \frac{\kappa(q)}{\psi(q)},
\end{equation}
for those $q$ where the left-hand side is finite;
where $\psi$ is the Laplace exponent belonging to the
Lamperti transform of $\Xb$'s cell process. Recall that
$\lvert u\rvert$ is the generation of $u$; that is, the length
of the word $u$.
It is known that $\kappa$ can be expressed
(see \cite[equation (5)]{BBCK-maps}) as
\[
  \kappa(q)
  = \psi(q) + \int_{(0,\infty)} c(x) \lvert 1-x\rvert^q \, \nu(\dd x).
\]

In general, if
$\omega\in(\dom\kappa)^\circ$ is such that $\kappa(\omega) \le 0$, then
the process
\[ 
  \Mm_\omega(n)
    = \sum_{u\in \tree,\, \lvert u\rvert = n+1} \Xx_u(0)^\omega,
    \quad n\ge 0, 
\]
is a supermartingale
in the filtration $\GG_n = \sigma(\Xx_u, \lvert u\rvert \le n)$;
and if $\kappa(\omega) = 0$, then it is a martingale.
This is shown in \cite[§2.3]{BBCK-maps},
and can be seen from \eqref{e:kappa-gen}, recalling that $\psi < \kappa$.
We define a change of measure
\[
  \left. \frac{\dd \PP^\omega_x}{\dd \PP_x} \right\rvert_{\GG_n}
  = x^{-\omega} \Mm_\omega(n),
\]
and additionally single out a distinguished cell $U = (U(n) : n\ge 0)$ 
under $\PP^\omega_\cdot$ by
\[
  \PP^\omega_x(U(n+1) = u \mid \GG_n)
  = \frac{\Xx_u(0)^\omega}{\Mm_\omega(n)},
  \quad u \in \tree, \, \lvert u\rvert = n+1.
\]
If we now define, by slight abuse of notation, $U(t)$ to be the
value of $U(n)$ such that $b_{U(n)} \le t < b_{U(n+1)}$,
then the (sub-Markov) process
\[ Y_\omega(t) = \Xx_{U(t)}(t-b_{U(t)}), \quad t \ge 0, \]
plays a special role, and is called the \emph{spine}.
It is a pssMp with index $-\alpha$ whose Lamperti transformation
has Laplace exponent $q \mapsto \kappa(\omega+q)$.

In the case of $\Zb\crc$, let us denote its cumulant by
$\kappa\crc$, an explicit formula for which was found in
\cite[equation (19)]{BBCK-maps}:
\[
  \kappa\crc(q)
  %=
  %\frac{\cos(\pi(q-\theta))}{\sin(\pi(q-2\theta))}\frac{\Gamma(q-\theta)}{\Gamma(q-2\theta)}
  =
  -\Gamma(1+2\theta-q)\Gamma(q-\theta) \frac{\cos(\pi(\theta+1-q))}{\pi}
  =
  -\frac{\Gamma(1+2\theta-q)\Gamma(q-\theta)}
  {\Gamma(\frac{3}{2}+\theta-q)\Gamma(-\frac{1}{2}-\theta+q)},
\]
for $q \in (\theta,1+2\theta)$.
If we define
\begin{align*}
  \omega_- &= \theta+1/2, \\
  \omega_0 &= \theta+1, \\
  \omega_+ &= \theta+3/2,
\end{align*}
then $\kappa\crc(\omega_-) = \kappa\crc(\omega_+) = 0$, and
$\kappa\crc(\omega_0) < 0$, and the corresponding spines are identifiable.
$Y\crc_{\omega_0}$ 
is a stable process with parameters $\theta,\rho$ such that
$\theta(1-\rho) = 1/2$, killed upon going below zero;
$Y\crc_{\omega_-}$ is the same process
conditioned to hit zero continuously
in the sense of \cite{CR-cond-zero}; and $Y\crc_{\omega_+}$
is the process conditioned to stay positive in the sense of
\cite{KRW-cond}.
The spines corresponding to $\omega_\pm$ were exploited by \cite{BBCK-maps}
to describe Boltzmann maps. 

Our novel growth-fragmentation process
is obtained by additionally introducing cells at the upward jump times of
cells in $\Zb\crc$.
When a cell jumps in mass from $u$ to $ux$, where $x>1$,
we introduce a new cell of mass $u(x-1)$
with probability $r$.
Write $\Zb$ for this process.
Mathematically, the process $\Zb$ is a Markovian growth-fragmentation
with the same cell process as $\Zb\crc$, but with selection probability
$c(x) = \Indic{x<1} + r\Indic{x>1}$, meaning that all downward jumps
lead to daughter cells, and upward jumps lead to daughter cells with
probability $r$.

In principle, the process $\Zb$ may not be well-defined, in that it
may not be possible to rank its elements in descending order. However,
this can be done provided that it is \emph{locally finite},
meaning that, at all times and for all compact sets $K\subset (0,\infty)$,
the number of cells with size in $K$ is finite.

Our work in the next section will be to prove that $\Zb$ is locally
finite, establish its cumulant and characterise its spines.

\begin{remark}
  \begin{enumerate}
    \item
      Our description of Markovian growth-fragmentations
      is a minor elaboration on 
      the setting of \cite{BeMGF,Shi-gf},
      in which
      only cell processes with negative jumps were considered,
      and the selection probability was absent.
      These
      changes do not affect the arguments there.

    \item
      Unlike in $\Zb\crc$, there is no preservation of total mass
      (i.e., $\ell^1$-norm) of $\Zb$
      at fragmentation events.
      The cells described above could be thought of as
      cells of negative initial mass $u(1-x)$ (which would conserve
      mass) `ricocheted' back above the origin.
      We will still refer to $\Zb$ as a growth-fragmentation, though perhaps
      it would be equally appropriate to call it a branching self-similar Markov process,
      à la \citet{BM-bLp}.

    \item
      The process $\Zb\crc$ is denoted $\mathbf{X}^{(-\theta)}_\theta$ in \cite{BBCK-maps}.
      That work also identifies the growth-fragmentation arising from the
      cycle perimeters of the metric balls in a Boltzmann map. The two
      processes are very similar, and differ only by a stopping-line
      time change which shifts the self-similarity index from $\alpha=-\theta$
      to $\alpha=1-\theta$.
  \end{enumerate}
\end{remark}

\section{Results and proofs}
\label{s:results}

\begin{theorem}
  \label{t:gf-rico}
  The process $\Zb$ is locally finite, is self-similar of index
  $-\theta$, and has cumulant given by
  \[
    \kappa(q)
    =
    -2^{\theta}
    \frac{\Gamma(\frac{1+2\theta-q}{2})\Gamma(\frac{2+2\theta-q}{2})}
    {\Gamma(\frac{1+b+\theta - q}{2})\Gamma(\frac{3-b+\theta-q}{2})}
    \frac{\Gamma(\frac{-\theta+q}{2})\Gamma(\frac{1-\theta+q}{2})}
    {\Gamma(\frac{1-b-\theta+q}{2})\Gamma(\frac{b-1-\theta+q}{2})},
    \quad q \in (\theta,1+2\theta),
  \]
  where $b \in (0,1/2]$ is chosen so that
  $
  r \sin (\pi(\theta-1/2))
  =
  \cos(\pi b)
  $.
%% Concordance with handwritten notes notation:
%   where $b = \abs{\nu-1}$ (i.e., $1-\nu$ when $\theta\le 1$ and $\nu-1$ when $\theta>1$), and
%   $\nu$ is such that $\sin \pi(\nu-1/2) = r\sin \pi(\theta-1/2)$,
%   with $\nu \in (1/2,1]$ when $\theta \in (1/2,1]$ and $\nu \in (1,3/2)$ when
%   $\theta \in (1,3/2)$.
  Define
  \begin{align*}
    \omega_- &= \theta+1-b, \\
    \omega_0 &= \theta+1, \\
    \omega_+ &= \theta+1+b.
  \end{align*}
  Then, $\omega_-$ and $\omega_+$ are the unique zeroes of $\kappa$,
  and $\kappa(q) < 0$ for $q\in (\omega_-,\omega_+)$.

  The spine $Y_{\omega_0}$ corresponding to $\omega_0$
  is the ricocheted stable process with
  parameters $(\theta,0,b)$;
  that is, $\theta(1-\rho) = 1/2$ and $\pf = r\sin(\pi(\theta-1/2))$.
  The spine $Y_{\omega_-}$ is the same process conditioned to
  reach zero continuously in the sense of \cite{CR-cond-zero},
  and the spine $Y_{\omega_+}$ is
  the process conditioned to avoid zero in the sense of \cite{KRW-cond}.
\end{theorem}
\begin{proof}
  The process $\Zb$, if it is indeed locally finite, has cumulant
  \[
    \kappa(q)
    =
    \kappa\crc(q) + r\int_{1}^\infty (x-1)^q\, \nu(\dd x),
  \]
  where $\nu$
  is the dislocation measure of $\Zb\crc$.

  Using a beta integral
  we find that
  \[
    \int_1^\infty (x-1)^{q-(\theta+1)} x^{-(\theta+1)}  \, \dd x
    = \frac{\Gamma(q-\theta)\Gamma(1-q+2\theta)}{\Gamma(1+\theta)},
    \quad q\in (\theta,1+2\theta),
  \]
  whence
  \begin{align}
    \kappa(q)
    &=
    \Gamma(1+2\theta-q)\Gamma(q-\theta)
    \frac{\cos(\pi b) - \cos(\pi(\theta+1-q))}{\pi} \nonumber \\
    &=
    -\frac{2}{\pi}
    \Gamma(1+2\theta-q)\Gamma(q-\theta)
    \sin\left(\pi\left(\frac{1+b+\theta-q}{2}\right)\right)
    \sin\left(\pi\left(\frac{b-1-\theta+q}{2}\right)\right) \label{e:kappa-alt} \\
    &=
    -2\pi
    \frac{\Gamma(1+2\theta-q)\Gamma(q-\theta)}
    {\Gamma(\frac{1+b+\theta-q}{2})\Gamma(\frac{3-b+\theta-q}{2})
    \Gamma(\frac{1-b-\theta+q}{2})\Gamma(\frac{b-1-\theta+q}{2})}, \nonumber
  \end{align}
  using product-sum identities and the reflection formula.
  The expression for $\kappa$ given in the statement then follows
  by applying the Legendre duplication formula to the
  gamma functions in the numerator.

  Treated as a meromorphic function, $\kappa$ has poles at
  \begin{align*}
    \rho_k              & = 2\theta + 1 + k,     & k & = 0,1,2, \dotsc, \\
    \hat{\rho}_k        & = \theta - k,          & k & = 0,1,2, \dotsc,
  \end{align*}
  and zeroes at
  \begin{align*}
    \zeta^{(1)}_k       & = b+\theta+2k+1,       & k & = 0,1,2,\dotsc,  \\
    \zeta^{(2)}_k       & = \theta-b+2k+3,       & k & = 0,1,2,\dotsc,  \\
    \hat{\zeta}^{(1)}_k & = b + \theta - (2k+1), & k & = 0,1,2,\dotsc,  \\
    \hat{\zeta}^{(2)}_k & = \theta - b - (2k-1), & k & = 0,1,2,\dotsc.
  \end{align*}
  Comparing these, we see that $\omega_- = \hat{\zeta}^{(2)}_0$
  and $\omega_+ = \zeta^{(1)}_0$ are the only
  zeroes of $\kappa$ within its domain $(\theta,1+2\theta)$.

  We address the local finiteness of $\Zb$. For this, it is enough
  to assume $r=1$, since other values of $r$ can be obtained by
  randomly removing cells from this one.
  Since there are values of $q$ such that $\kappa(q)\le0$,
  \cite[Lemma 2]{BeMGF} implies that $x \mapsto x^q$
  is excessive for the cell process $X$ of $\Zb$, in the sense that
  \[ \EE_x \left[ X_t^q + \sum_{0<s\le t} \lvert \Delta X_s\rvert^q \right] \le x^q. \]
  In turn, \cite[Theorem 1]{BeMGF} implies that $\Zb$ is locally finite.
  (These results are proven under the assumption that the cell process
  only jumps downward, but their statements remain valid in our more
  general setting, if we use the above definition of excessiveness.)

  The claim about the spine processes follows by comparing
  $q \mapsto \kappa(q+\omega_0)$ with the description
  of the ricocheted stable process in section~\ref{s:ricochet}.
  As shown in \cite{CR-cond-zero,KRW-cond}
  and summarised in \cite[\S 13.4.4]{Kyp2},
  the conditionings of the ricocheted
  stable process to reach zero continuously and avoid zero
  have Lamperti transformation with Laplace exponent
  $q\mapsto \kappa(q+\omega_-)$ and $q\mapsto\kappa(q+\omega_+)$,
  respectively.
\end{proof}

The relation $\pf = r\sin\pi(\theta-1/2)$ reflects the differing rates 
of ricochets, which occur as a result of upward jumps in the 
growth-fragmentation and downward jumps in the ricocheted stable process.

The expression of $\kappa$ as a ratio of eight gamma functions,
instead of as formula \eqref{e:kappa-alt},
may
seem complex,
but it makes much clearer the Wiener-Hopf
factorisations of the spine processes.
For instance, $Y_{\omega_0}$
is a pssMp whose Lamperti transformation
has Laplace exponent
\[
  \psi(q)
  =
  -2^{\theta}
  \frac{\Gamma(\frac{\theta-q}{2})\Gamma(\frac{1+\theta-q}{2})}
  {\Gamma(\frac{2-b-q}{2})\Gamma(\frac{b-q}{2})}
  \times
  \frac{\Gamma(\frac{1+q}{2})\Gamma(\frac{2+q}{2})}
  {\Gamma(\frac{2-b+q}{2})\Gamma(\frac{b+q}{2})},
\]
and the two functions on either side of the $\times$ represent its
Wiener-Hopf factors \cite{KPV-dblhg}.

When $\theta \to 1/2$ or $\theta = 3/2$, $b$ takes the value $1/2$ and
$\kappa$ approaches the cumulant of the Brownian fragmentation
\cite[p.~338]{Ber-ssf} or the exploration of a Boltzmann triangulation
\cite{BCK-maps}, respectively.

A particularly relevant situation appears in the following result,
part of which was observed, using properties of CLE rather than
the explicit description
of $\Zb$, in \cite[p.~33]{MSW-scle}.
\begin{corollary}
  When $r=1$, $b = \lvert \theta-1\rvert$.
  In the case $\theta > 1$, we have $\omega_- = 2$ and $\omega_+ = 2\theta$;
  in the case $\theta \le 1$, we have $\omega_- = 2\theta$ and $\omega_+ = 2$.
\end{corollary}

Specialising still further, we have a relation with planar
Brownian motion.

\begin{corollary}
  When $\theta = 1$ and $r=1$, the cell process of $\Zb$ is the same
  as that found by \citet{AdS-exc}, where it describes the size
  of excursions in a planar Brownian motion.
  The process $\Zb$ itself can be obtained from the process
  $\overline{\mathbf{X}}$ in that work by taking absolute values
  of the cell `masses'.
  Its cumulant is given by
  \[
    \kappa(q)
    =
    -2
    \frac{\Gamma(\frac{3-q}{2})\Gamma(\frac{q-1}{2})}
    {\Gamma(\frac{2-q}{2})\Gamma(\frac{q-2}{2})},
    \quad q \in (1,3).
  \]

  Denote by $\mathcal{Z}$ the generational growth-fragmentation
  associated with $\Zb$.
  In this case, $\omega_- = \omega_0 = \omega_+ = 2$, the process
  \[
    \Mm_2(n) = \sum_{u\in \tree, \lvert u\rvert = n+1} 
    \mathcal{Z}_u(0)^2, 
    \quad n\ge 0,
  \]
  is a martingale, and the corresponding spine is the absolute value
  of a symmetric Cauchy process.
\end{corollary}

The second part of this corollary provides a proof of
\cite[Proposition 3.7]{AdS-exc}, once the identification of
the growth-fragmentations is known.
However, the process $\overline{\mathbf{X}}$
constructed in that work contains more information than the
approach using $\Zb$; the ricochet step loses some information
about the `sign' of excursions. This information
is preserved (and the setting extended to cover
excursions involving stable processes) by Da Silva's
framework
of signed growth-fragmentations
\cite{DS-signed}.
It must also be said that 
ricocheted processes are not strictly required for this
application: when $\theta=1$ we are in a symmetric setting,
so $\Zb$ can be described in terms of the 
absolute value of a Cauchy process
\cite{CPP-radial,KKPW-T0}.

Using this description, it is possible to learn something
about the appearance of small excursions. In the generational
growth-fragmentation $\mathcal{Z}$, the cells $u$ describe
excursions in the upper half-plane above some
barrier, represented by time.
For each $u\in\tree$ and $t\ge 0$, let $u(t)$
be the most recent ancestor of $u$
which was alive at time $t$ (which may be $u$ itself.)
If we define
\[
  T^-_u(\epsilon) 
  = \inf\{ t\ge 0: \mathcal{Z}_{u(t)}(t-b_{u(t)}) < \epsilon\},
  \quad u \in \tree,
\]
then 
$\mathcal{Z}(T^-(\epsilon)) \coloneqq (\mathcal{Z}_{u(T^-_u(\epsilon))}(T^-_u(\epsilon)-b_{u(T^-_u(\epsilon))}) : u \in \tree, u(T^-_u(\epsilon))=u)$
represents the collection of excursions frozen at the first time that their
size drops below $\epsilon$. 
(The condition `$u(T^-_u(\epsilon)) = u$' avoids multiple counting.)
Additive functionals of the
set can be described using the growth-fragmentation:
\begin{corollary}
  \label{c:cauchy}
  Let $\theta=1$ and $r=1$.
  For an integrable function $f$ and $\epsilon>0$,
  \[
%    \EE_1 \left[ \sum_{u\in\tree,u(T_u(\epsilon))=u}
%    f(\mathcal{Z}_u(T_u(\epsilon)-b_{u(T_u(\epsilon))})) \right]
    \EE_1\! \left[ \sum_{z \in \mathcal{Z}(T^-(\epsilon))} f(z)z^2 \right]
    =
    \int_{0}^\epsilon \frac{2}{\pi}
    \frac{1}{1-y^2} \sqrt{\frac{1-\epsilon^2}{\epsilon^2-y^2}}
    f(y) \, \dd y.
  \]
\end{corollary}
\begin{proof}
  Using a version of \cite[Proposition 4.1]{BBCK-maps},
  and writing $g(z) = f(z)z^{2}$,
  we obtain
  \begin{align*}
    \EE_1\!\! \left[ \sum_{u\in\tree, u(T^-_u(\epsilon)) = u} g(\mathcal{Z}_u(T^-_u(\epsilon)-b_{u(T^-_u(\epsilon))}) \right]
%     =
%     \EE_1 \left[ \sum_{u\in\tree} f(\mathcal{Z}_u(T_u(\epsilon)))
%       \mathcal{Z}_u(T_u(\epsilon)^{-2}
%       \mathcal{Z}_u(T_u(\epsilon)^2
%     \right]
    &=
    \EE_1\!\! \left[
      \sum_{u \in \tree}
      g(\mathcal{Z}_u(t)) 
      \Indic{\forall s \le t: \mathcal{Z}_{u(s)}(s) \ge \epsilon}
      \Indic{\mathcal{Z}_{u}(t)<\epsilon}
    \right] \\
    &=
    \EE^{\omega_0}_1\bigl[ g(Y_{T^-(\epsilon)}) Y_{T^-(\epsilon)}^{-2} \bigr]
    =
    \EE^{\omega_0}_1\bigl[ f(Y_{T^-(\epsilon)}) \bigr],
  \end{align*}
  where $Y$ is the spine process corresponding to exponent $\omega_0 = 2$,
  and $T^-(\epsilon)$ is its first passage time below $\epsilon$.
  Since $Y$ is the absolute value of a symmetric Cauchy process,
  the result follows using the hitting distribution in \cite{BGR-st-hit}.
\end{proof}

We can also observe a variant of \cite[Proposition~1.12]{Ber-fc}:
\[
  \lim_{\epsilon \downarrow 0}
  \EE_1\!
  \left[
    \sum_{z \in \mathcal{Z}(T^-(\epsilon))} f(z/\epsilon) z^2 
%     \sum_{u\in\tree} f_{\epsilon}(\mathcal{Z}_u(T_u(\epsilon))) \mathcal{Z}_u(T_u(\epsilon))^2
  \right]
  =
  \int_0^1 \frac{2f(y)\,\dd y}{\pi\sqrt{1-y^2}}.
\]

In principle, it is possible to prove similar results for the process
$\Zb\crc$, using \cite{Rog-hit}, or $\Zb$ with $\theta\ne 1$, using
\cite[Theorem 2.1]{KPV-dblhg}. The latter is probably the
most interesting for this note, but the
analogue of the right-hand side in \autoref{c:cauchy} involves
an integral over a product of sums of hypergeometric functions,
and no simplification appears to be possible.
However, by a remarkable coincidence of distributions,
it is possible to obtain the first passage upward.

Define the first passage time
\[
  T^+_u(x) = \inf\{ t\ge 0 : \mathcal{Z}_{u(t)}(t-b_{u(t)}) > x \},
  \quad u \in \tree,
\]
of the size of cell $u$ above level $x$, and $\mathcal{Z}(T^+(x))$ analogously
with the downward passage case.
We allow that if the size of a cell never exceeds level $x$ before its death,
it is omitted from $\mathcal{Z}(T^+(x))$.
Additive functionals of the growth-fragmentation frozen as the cells
exceed size $x$ are given as follows.

\begin{corollary}
  Let $r = 1$, $x>1$ and let $f$ be an integrable function. Then
  \[
    \EE_1\!\left[ \sum_{z \in \mathcal{Z}(T^+(x))} f(z) z^2 \right]
    = \int_x^\infty f(y) h_\theta(x,y) \, \dd y
  \]
  where, if $\theta\le 1$,
  \[
    h_\theta(x,y) = 2\frac{\sin \pi(\theta-1/2)}{\pi}
    (x^2-1)^{\theta-1/2} (y^2-x^2)^{1/2-\theta} \frac{y}{y^2-1}
    , \quad y > x,
  \]
  and if $\theta>1$,
  \begin{align*}
    h_\theta(x,y)
    &=
    2\frac{\sin \pi(\theta-1/2)}{\pi}
    (x^2-1)^{\theta-1/2} (y^2-x^2)^{1/2-\theta} \frac{y}{y^2-1} 
    \\
    & \quad\  {} - 2(\theta-1)\frac{\sin \pi(\theta-1/2)}{\pi}
    \frac{x}{y} (y^2-x^2)^{1/2-\theta} \int_1^{x^2} (t-1)^{\theta-3/2} t^{-1/2}\, \dd t,
    \quad y>x.
  \end{align*}
\end{corollary}
\begin{proof}
  We begin with the $\theta \le 1$ case,
  and recall first that when $r=1$, we have $\omega_+ = 2$, and the corresponding
  spine is the ricocheted stable process conditioned to stay positive,
  which we will denote $Y$ under measure $\PP^{\omega_+}_\cdot$.
  Analogously with the proof of the
  preceding corollary, we can express
  \[
    \EE_1\! \left[  \sum_{z \in \mathcal{Z}(T^+(x))} f(z) z^2 \right]
    = \EE_1^{\omega_+}[ f(Y_{T^+(x)})],
  \]
  where $T^+(x)$ is the first passage time of $Y$ above $x$.
  The process $Y$ can be studied via its Lamperti transform, say $\xi$,
  which has Laplace exponent
  \[
    \psi_+(q)
    = \kappa(q+2)
    = -2^\theta
    \frac{\Gamma(\frac{2\theta-1-q}{2})}
    {\Gamma(\frac{-q}{2})}
    \frac{\Gamma(\frac{2-\theta+q}{2})\Gamma(\frac{3-\theta+q}{2})}
    {\Gamma(\frac{2+q}{2})\Gamma(\frac{2-2\theta+q}{2})}.
  \]
  It is convenient to rescale time and space by defining
  $\xi'_t = 2\xi_{2^{-\theta}t}$, which has Laplace exponent
  $\psi'(q) = 2^{-\theta}\psi_+(2q)$.
  We see directly that
  \[
    \EE_1^{\omega_+}[ f(Y_{T^+(x)})]
    = \EE_0[ f(e^{H_{S^+(2\log x)}/2})],
  \]
  where $H$ is the ascending ladder height process of $\xi'$ and
  $S^+(\cdot)$ its first passage time upward.

  Examining $\psi'$, we see that
  the Laplace exponent
  of the ascending ladder height process of $\xi'$ is given by
  \[ 
    q \mapsto \frac{\Gamma(\theta - 1/2 + q)}{\Gamma(q)} , \quad q \ge 0. 
  \]
  This is identical to the ascending ladder height of
  the Lamperti transform of the path-censored
  stable process \cite[section 5.4]{KPW-cens}, though we stress that the descending
  ladder height processes differ.
  This allows us to express
  \[
    \EE_0[ f(e^{H_{S^+(2\log x)}/2})]
    = \EE_1\bigl[ f\bigl(X_{T^+(x^2)}^{1/2}\bigr)\bigr]
    = \EE_{x^{-2}} \bigl[ f\bigl(x X_{T^+(1)}^{1/2}\bigr)\bigr],
  \]
  where
  $X$ is the stable process with index
  $\theta$ and positivity parameter $1-\frac{1}{2\theta}$,
  and $T^+(\cdot)$ is its first passage time upward.
  Putting together the pieces, we have that
  \[
    \EE_1\! \left[  \sum_{z \in \mathcal{Z}(T^+(x))} f(z) z^2 \right]
    = \int_1^\infty f(xz^{1/2}) g(z) \, \dd z
    = \int_1^\infty f(y) 2y x^{-2} g(y^2x^{-2})\, \dd y,
  \]
  where $g(y) = \PP_{x^{-2}}(X_{T^+(1)} \in \dd y) / \dd y$.
  This quantity was calculated by \citet{Rog-hit} and appears as
  equation (3) in \cite{KPW-cens}. Substituting $g$ completes the proof.

  The proof for the case $\theta>1$ is similar, with the key difference
  being that now $\omega_- = 2$, so we make use of the spine given by
  the ricocheted stable process conditioned to hit zero,
  with measures $\PP^{\omega_-}_\cdot$, and obtain
  \begin{equation}\label{e:gt1}
    \EE_1\! \left[  \sum_{z \in \mathcal{Z}(T^+(x))} f(z) z^2 \right]
    = \EE_1^{\omega_-}[ f(Y_{T^+(x)}) \Indic{T^+(x) < T^0}]
    = \EE_{x^{-2}}\bigl[ f(xX_{T^+(1)}^{1/2}) \Indic{T^+(1) < T^0}\bigr]
  \end{equation}
  where $T^0$ is the time at which $Y$ (in the middle expression)
  or $X$ (in the latter) hits zero.
  In deriving this equality, we follow the same proof structure, the difference
  being that we use section 5.5 in \cite{KPW-cens}, corresponding to our
  value of $\theta$.
  In order to evaluate the right-hand side of \eqref{e:gt1}, we use
  Theorem 1.5 in \cite{KPW-cens}, and this gives the expression in the statement.
\end{proof}

The preceding proof relies on the curious fact that, up to a space
transformation, the ricocheted stable process conditioned to stay positive
(when $\theta\le 1$) or to hit zero (when $\theta>1$) attains new maxima
in the same way as the stable process (killed upon reaching zero).
We are not aware of any pathwise explanation of this property.

%\bibliographystyle{abbrvnat}
%\bibliography{master}

\end{document}